\documentclass[11pt]{amsart}
\usepackage{amsopn}
\usepackage{amssymb, amscd}
\usepackage{multirow}
\usepackage{graphicx, graphics, epsfig}
\usepackage{faktor} 
\usepackage{enumerate}
\usepackage{tikz}
\usepackage{pgfplots}
\usepackage{hyperref}
\usepackage{color}
\usepackage{dynkin-diagrams}

\newcommand{\nc}{\newcommand}

\nc{\fg}{\mathfrak{f} } \nc{\vg}{\mathfrak{v} } \nc{\wg}{\mathfrak{w} }
\nc{\zg}{\mathfrak{z} } \nc{\ngo}{\mathfrak{n} } \nc{\kg}{\mathfrak{k} }
\nc{\mg}{\mathfrak{m} } \nc{\bg}{\mathfrak{b} } \nc{\ggo}{\mathfrak{g} } \nc{\eg}{\mathfrak{e} }
\nc{\ggob}{\overline{\mathfrak{g}} } \nc{\sog}{\mathfrak{so} }
\nc{\sug}{\mathfrak{su} } \nc{\spg}{\mathfrak{sp} } \nc{\slg}{\mathfrak{sl} }
\nc{\glg}{\mathfrak{gl} } \nc{\cg}{\mathfrak{c} } \nc{\rg}{\mathfrak{r} }
\nc{\hg}{\mathfrak{h} } \nc{\tg}{\mathfrak{t} } \nc{\ug}{\mathfrak{u} }
\nc{\dg}{\mathfrak{d} } \nc{\ag}{\mathfrak{a} } \nc{\pg}{\mathfrak{p} }
\nc{\sg}{\mathfrak{s} } \nc{\affg}{\mathfrak{aff} } \nc{\qg}{\mathfrak{q} } \nc{\lgo}{\mathfrak{l} }

\nc{\pca}{\mathcal{P}} \nc{\nca}{\mathcal{N}} \nc{\lca}{\mathcal{L}}
\nc{\oca}{\mathcal{O}} \nc{\mca}{\mathcal{M}} \nc{\tca}{\mathcal{T}}
\nc{\aca}{\mathcal{A}} \nc{\cca}{\mathcal{C}} \nc{\gca}{\mathcal{G}}
\nc{\sca}{\mathcal{S}} \nc{\hca}{\mathcal{H}} \nc{\bca}{\mathcal{B}}
\nc{\dca}{\mathcal{D}} \nc{\eca}{\mathcal{E}} \nc{\wca}{\mathcal{W}} \nc{\ica}{\mathcal{I}} \nc{\kca}{\mathcal{K}}

\nc{\vp}{\varphi} \nc{\ddt}{\tfrac{d}{dt}} \nc{\dsdt}{\tfrac{d^2}{dt^2}} \nc{\dds}{\frac{d}{ds}}
\nc{\dpar}{\frac{\partial}{\partial t}} \nc{\dpars}{\tfrac{\partial}{\partial t}} \nc{\im}{\mathrm{i}}

\nc{\SO}{\mathrm{SO}} \nc{\Spe}{\mathrm{Sp}} \nc{\Sl}{\mathrm{SL}}
\nc{\SU}{\mathrm{SU}} \nc{\Or}{\mathrm{O}} \nc{\U}{\mathrm{U}} \nc{\Gl}{\mathrm{GL}}
\nc{\Se}{\mathrm{S}} \nc{\Cl}{\mathrm{Cl}} \nc{\Spin}{\mathrm{Spin}}
\nc{\Pin}{\mathrm{Pin}} \nc{\G}{\mathrm{GL}_n(\RR)} \nc{\g}{\mathfrak{gl}_n(\RR)}
\nc{\Eg}{\mathrm{E}} \nc{\Fg}{\mathrm{F}} \nc{\Gg}{\mathrm{G}}

\nc{\RR}{{\Bbb R}} \nc{\HH}{{\Bbb H}} \nc{\CC}{{\Bbb C}} \nc{\ZZ}{{\Bbb Z}}
\nc{\FF}{{\Bbb F}} \nc{\NN}{{\Bbb N}} \nc{\QQ}{{\Bbb Q}} \nc{\PP}{{\Bbb P}} \nc{\OO}{{\Bbb O}}

\nc{\vs}{\vspace{.2cm}} \nc{\vsp}{\vspace{1cm}} \nc{\ip}{\langle\cdot,\cdot\rangle}
\nc{\ipp}{(\cdot,\cdot)} \nc{\la}{\langle} \nc{\ra}{\rangle} \nc{\unm}{\tfrac{1}{2}}
\nc{\unc}{\tfrac{1}{4}} \nc{\und}{\frac{1}{16}} \nc{\no}{\vs\noindent}
\nc{\lam}{\rho^2(\RR^n)^*\otimes\RR^n} \nc{\tangz}{{\rm T}^{\rm Zar}}
\nc{\nor}{{\sf n}}  \nc{\mum}{/\!\!/} \nc{\kir}{/\!\!/\!\!/}
\nc{\Ri}{\tfrac{4\Ric_{\mu}}{||\mu||^2}} \nc{\ds}{\displaystyle}
\nc{\ben}{\begin{enumerate}} \nc{\een}{\end{enumerate}} \nc{\f}{\frac}
\nc{\lb}{[\cdot,\cdot]} \nc{\isn}{\tfrac{1}{||v||^2}}
\nc{\gkp}{(\ggo=\kg\oplus\pg,\ip)} \nc{\ukh}{(\ug=\kg\oplus\hg,\ip)}
\nc{\tgkp}{(\tilde{\ggo}=\kg\oplus\pg,\ip)}
\nc{\wt}{\widetilde}
\nc{\iop}{\mathtt{i}} \nc{\jop}{\mathtt{j}} 
\nc{\Hk}{H_{\kil}} \nc{\gk}{g_{\kil}}

\nc{\Hess}{\operatorname{Hess}} \nc{\ad}{\operatorname{ad}}
\nc{\Ad}{\operatorname{Ad}} \nc{\rank}{\operatorname{rk}}
\nc{\Irr}{\operatorname{Irr}} \nc{\End}{\operatorname{End}}
\nc{\Aut}{\operatorname{Aut}} \nc{\Inn}{\operatorname{Inn}}
\nc{\Der}{\operatorname{Der}} \nc{\Ker}{\operatorname{Ker}}
\nc{\Iso}{\operatorname{Iso}} \nc{\Diff}{\operatorname{Diff}}
\nc{\Lie}{\operatorname{L}} \nc{\tr}{\operatorname{tr}} \nc{\dif}{\operatorname{d}}
\nc{\sen}{\operatorname{sen}} \nc{\modu}{\operatorname{mod}}
\nc{\CRic}{\operatorname{PP}} \nc{\Cric}{\operatorname{P}} \nc{\Ricci}{\operatorname{Ric}}
\nc{\sym}{\operatorname{sym}} \nc{\herm}{\operatorname{herm}} \nc{\symac}{\operatorname{sym^{ac}}}
\nc{\symc}{\operatorname{sym^{c}}} \nc{\scalar}{\operatorname{Sc}}
\nc{\grad}{\operatorname{grad}} \nc{\ricci}{\operatorname{Rc}} \nc{\kil}{\operatorname{B}} \nc{\cas}{\operatorname{C}} \nc{\lic}{\operatorname{L}}
\nc{\Nor}{\operatorname{Norm}}  \nc{\ricc}{\operatorname{Rc^{c}}}
\nc{\Ricc}{\operatorname{Ric^{c}}} \nc{\ricac}{\operatorname{Rc^{ac}}}
\nc{\Ricac}{\operatorname{Ric^{ac}}} \nc{\Riem}{\operatorname{Rm}} \nc{\Sec}{\operatorname{Sec}}
\nc{\riccig}{\operatorname{ric^{\gamma}}} \nc{\mm}{\operatorname{m}} \nc{\Mm}{\operatorname{M}}
\nc{\Le}{\operatorname{L}} \nc{\tang}{\operatorname{T}}
\nc{\level}{\operatorname{level}} \nc{\rad}{\operatorname{r}}
\nc{\abel}{\operatorname{ab}} \nc{\CH}{\operatorname{CH}} \nc{\Cone}{{\mathcal C}} \nc{\CCone}{\operatorname{CC}} \nc{\CP}{{\mathcal P}}
\nc{\mcc}{\operatorname{mcc}} \nc{\Adj}{\operatorname{Adj}}
\nc{\Order}{\operatorname{O}}  \nc{\inj}{\operatorname{inj}} \nc{\proy}{\operatorname{pr}}
\nc{\vol}{\operatorname{vol}} \nc{\Diag}{\operatorname{Dg}} \nc{\Diagg}{\operatorname{Diag}}
\nc{\Spec}{\operatorname{Spec}} \nc{\Ima}{\operatorname{Im}} \nc{\Rea}{\operatorname{Re}}
\nc{\spann}{\operatorname{span}} \nc{\Aff}{\operatorname{Aff}} \nc{\E}{\operatorname{E}} \nc{\id}{\operatorname{id}} \nc{\dete}{\operatorname{det}} \nc{\Crit}{\operatorname{Crit}} \nc{\val}{\operatorname{val}}

\theoremstyle{plain}
\newtheorem{theorem}{Theorem}[section]
\newtheorem{proposition}[theorem]{Proposition}
\newtheorem{corollary}[theorem]{Corollary}
\newtheorem{lemma}[theorem]{Lemma}

\theoremstyle{definition}

\theoremstyle{remark}
\newtheorem{remark}[theorem]{Remark}
\newtheorem{example}[theorem]{Example}

\title{Pluriclosed metrics on compact semisimple Lie groups}

\author{Jorge Lauret}  \author{Facundo Montedoro}

\address{FaMAF, Universidad Nacional de C\'ordoba and CIEM, CONICET (Argentina)}
\email{jorgelauret@unc.edu.ar} \email{facundo.montedoro@mi.unc.edu.ar} 

\thanks{This research was partially supported by three grants from, respectively,  CONICET, Univ. Nac. de C\'ordoba and Foncyt (Argentina)}

\date{\today}

\begin{document}

\maketitle

\begin{abstract}
Given a compact semisimple Lie group $G$ and a maximal torus $T\subset G$, we give an explicit description of all left and $\Ad(T)$-invariant pluriclosed Hermitian structures on $G$ in terms of the corresponding root system.  They depend on $2d+1$ parameters in the irreducible case, where $\dim{T}=2d$.  As applications, we obtain that the only left and $\Ad(T)$-invariant pluriclosed metrics which are also CYT are bi-invariant metrics (i.e., Bismut flat) and study the pluriclosed flow as a neat ODE system.    
\end{abstract}

\tableofcontents

\section{Introduction}\label{intro}

A Hermitian metric $g$ on a complex manifold $(M,J)$ is said to be {\it pluriclosed} (or SKT, short for strong K\"ahler with torsion) if the K\"ahler form $\omega=g(J\cdot,\cdot)$ satisfies that $dd^c\omega=0$ (we refer to the recent surveys \cite{Fin,FinGrn2} for further information).  Being $dd^c\omega$ a $4$-form, the condition turns out to be very strong in high dimensions, specially in the homogeneous case (see \cite{ArrLfn, ArrNcl, FinPrd, FrbSwn} for structure and classification results on solvmanifolds).  The primal compact examples are bi-invariant metrics on any even dimensional compact semisimple Lie group $G$ endowed with any left-invariant complex structure.  As known, these are precisely Bismut flat manifolds with finite fundamental group (see \cite{WngYngZhn}).  It was proved in \cite{FinGrnVzz} that a compact homogeneous complex manifold with finite fundamental group (i.e., a C-space) admits an invariant pluriclosed metric only if it is a Lie group (modulo a flag manifold, i.e., a K\"ahler homogeneous manifold). 

Any left-invariant complex structure $J$ on a compact semisimple Lie group $G$ is determined by the choices of a maximal torus $T\subset G$ and of a set of positive roots $\Delta^+$ (see \cite{Sml,Wng,Ptt}), they are often called {\it Samelson} complex structures.  More precisely, $J$ leaves invariant the $\kil_\ggo$-orthogonal decomposition $\ggo=\tg\oplus\qg$ of the Lie algebra $\ggo$ of $G$, where $\kil_\ggo$ is the Killing form of $\ggo$, $\tg$ is the Lie algebra of $T$, $J_\tg$ is any linear map such that $J_\tg^2=-I$ and on the complexification $\qg^c=\bigoplus\limits_{\alpha\in\Delta} \CC E_\alpha$, the complex linear map determined by $J_\qg$ on $\qg^c$ is given by $J_\qg E_\alpha=\pm\im E_\alpha$ for all $\alpha\in\pm\Delta^+$.  The space of all left-invariant complex structures on $G$ therefore depends on $2d^2$ parameters up to biholomorphism if $\dim{T}=2d$.  Any left and $\Ad(T)$-invariant metric which is compatible with a fixed $J=(J_\tg,J_\qg)$ is of the form $g=g_\tg+g_\qg$, where $g_\tg$ is any inner product on $\tg$ compatible with $J_\tg$ and $g_\qg=(x_\alpha)_{\alpha\in\Delta^+}$ is given by 
$$
g_\qg=\sum_{\alpha\in\Delta^+} x_\alpha (-\kil_\ggo)|_{\qg_\alpha}, \qquad x_\alpha>0, \quad
\qg_\alpha:=(\CC E_\alpha\oplus\CC E_{-\alpha})\cap\ggo,
$$
producing a space of Hermitian metrics that depends on $d^2+|\Delta^+|$ parameters.   

The following strong results on left-invariant pluriclosed metrics on $G$ were obtained by Fino and Grantcharov in \cite{FinGrn1}:
\begin{enumerate}[(a)]
%\item If $(G,J)$ admits a pluriclosed left-invariant metric, then it also admits a left and $\Ad(T)$-invariant pluriclosed metric $g$ for the maximal torus $T$ of $G$ attached to $J$.  
%
\item For any left and $\Ad(T)$-invariant pluriclosed metric $g$, there exists a bi-invariant metric $g_b$ on $G$ such that $g_\tg=g_b|_\tg$.  The proof is based on cohomological properties of the torus fibration $G\rightarrow G/T$ over the full flag manifold $G/T$.    

\item A $5$-parametric space of pluriclosed metrics on $G=\SO(9)$ is given for certain complex structure $J$.   

\item The only left and $\Ad(T)$-invariant pluriclosed metrics which are also CYT (Calabi-Yau with torsion, i.e., the Bismut Ricci form $\rho^B$ vanishes, see \S\ref{CYT-sec}) are bi-invariant metrics.  
\end{enumerate}
%As far as we know, beyond $\SU(2)\times\SU(2)$, $\SU(3)$ and $\SO(9)$, the only other known examples of pluriclosed left-invariant metrics (non bi-invariant) on compact semisimple Lie groups were exhibited in \cite{Phm} for $G=\Gg_2$.  
Explicit examples of pluriclosed left-invariant metrics (non bi-invariant) were also exhibited in \cite{Phm} for $G=\Gg_2$.  

\begin{remark}\label{refe}
It is worth recalling that given a pluriclosed metric $\omega$ on a compact manifold, $\omega+\overline{\partial}\psi+\partial\overline{\psi}$ is also a pluriclosed metric for any sufficiently small $(1,0)$-form $\psi$.  Moreover, using that the $(1,1)$-Aeppli cohomology is one-dimensional (see \cite[Theorem 3.1]{Brb}), one obtains that this construction gives all left-invariant pluriclosed metrics on a compact simple Lie group.  \end{remark} 

In this paper, after giving a formula for $dd^c\omega$ in terms of roots for any left and $\Ad(T)$-invariant Hermitian structure on $G$ (see Lemma \ref{ddco}), we provide an alternative algebraic proof for the result stated in (a) above (see Proposition \ref{SKT3}) and then we prove the following complete classification (see Theorem \ref{SKT2} for a more detailed statement).  We fix a left-invariant complex structure $J$ attached to a maximal torus $T$ and consider the decomposition $\tg=\tg_1\oplus\dots\oplus\tg_s$, where the $\tg_i$ are the corresponding maximal tori of the simple factors of $G$.  

\begin{theorem}\label{SKT-intro}
On a compact semisimple Lie group $G$ endowed with a Samelson complex structure $J$, any left and $\Ad(T)$-invariant pluriclosed metric is given as follows: $g_\tg(\tg_i,\tg_j)=0$ for all $i\ne j$ and on each simple factor, up to scaling, $g_\tg=-\kil_\ggo|_\tg$ and $g_\qg=(x_\alpha)_{\alpha\in\Delta^+}$, where
$$
x_\alpha:=\sum_{i=1}^n k^\alpha_ix_i -\left(\sum_{i=1}^n k^\alpha_i -1\right), \qquad \forall \alpha=\sum_{i=1}^n k^\alpha_i\alpha_i\in\Delta^+, \quad k^\alpha_i\in\ZZ_{\geq 0},  
$$
$\Pi=\{\alpha_1,\dots,\alpha_n\}\subset\Delta^+$ is the set of simple roots and $x_i:=x_{\alpha_i}$ for $i=1,\dots,n$ (equivalently, $x_{\alpha+\beta}=x_\alpha+x_\beta-1$ for all $\alpha,\beta,\alpha+\beta\in\Delta^+$).       
\end{theorem}

If $\rank(G)=2d$, then each of these metrics is compatible with a $d(d-1)$-parametric space of left-invariant complex structures on $G$.  According to the theorem, for a fixed $J$, if the space of all left and $\Ad(T)$-invariant pluriclosed metrics compatible with $J$ is non-empty, then it depends on $r(2d+1)$ parameters and contains an $r$-parametric space of bi-invariant metrics, where $r$ is the number of irreducible factors of $(G,J)$ as a complex manifold.  In particular, any irreducible $J$ admits at most one compatible bi-invariant metric up to scaling.  Note that $J$ can be irreducible on a non-simple $G$ (see Example \ref{G1G2}).  

Intriguingly enough, the last line in the theorem should be compared with the K\"ahler condition for the Hermitian structure $(J_\qg,g_\qg)$ on the flag manifold $G/T$, given by $x_{\alpha+\beta}=x_\alpha+x_\beta$ for all $\alpha,\beta,\alpha+\beta\in\Delta^+$, whose solutions are precisely 
$$
x_\alpha:=\sum\limits_{i=1}^{n} k^\alpha_ix_i, \qquad\forall\alpha=\sum\limits_{i=1}^{n}k^\alpha_i\alpha_i\in\Delta^+.
$$  
We do not know whether there may be some hidden geometric reason for this behavior.    

As a first application of the explicit description of pluriclosed metrics given in Theorem \ref{SKT-intro}, we provide in \S\ref{CYT-sec} an alternative proof for the result stated in (c) above using the convexity properties of the functional 
\begin{equation}\label{F-intro}
F(g)=F(x_1,\dots,x_n):=\sum_{\alpha\in\Delta^+} x_\alpha-\log{x_\alpha}. 
\end{equation}
Secondly, we apply Theorem \ref{SKT-intro} to study the pluriclosed flow $\dpars\omega=-(\rho^B)^{1,1}$ for a one-parameter family $\omega(t)=g(t)(J\cdot,\cdot)$ as an ODE system for 
$$
X(t):=(x_1(t),\dots,x_{n}(t))\in\RR^n. 
$$  
The pluriclosed flow was introduced in \cite{StrTian} and in the homogeneous case, it has been studied mostly on solvmanifolds (see \cite{ArrLfn,FusVzz} and the references therein).  On a compact semisimple Lie group $G$, using a general convergence result from \cite{GrcJrdStr}, Barbaro proved in \cite{Brb} the convergence to a Bismut flat metric starting at any pluriclosed metric (not necessarily left-invariant) by computing the $(1,1)$-Aeppli cohomology of $G$, obtaining as an application the non-existence of pluriclosed and CYT metrics on $G$ other than the bi-invariant ones.  

We show in \S\ref{PF-sec} that the metrics $g(t)$ attached to $X(t)$ pluriclosed flow according to the following ODE system: 
\begin{equation}\label{PF2} 
x_j'
%=-\sum_{\substack{\alpha\in\Delta^+\\ 1\leq i\leq n}} 
=\sum_{\alpha\in\Delta^+,1\leq i\leq n} \left(1-\tfrac{1}{x_\alpha}\right) k^\alpha_i \la\alpha_i,\alpha_j\ra, \qquad\forall j=1,\dots,n,     
\end{equation}
or equivalently,
\begin{equation}\label{PF-intro}
X'(t)=-Q\grad(F)_{X(t)}, \qquad \mbox{where}\quad Q:=[\la\alpha_i,\alpha_j\ra],     
\end{equation}
and $F$ is the functional defined in \eqref{F-intro}.  Thus $Y(t):=Q^{-\unm}X(t)$ is the negative gradient flow of the real analytic function $H(Y):=F(Q^{\unm}Y)$, which implies that any solution $X(t)$ converges to $(1,\dots,1)$ as $t\to\infty$.  Any left and $\Ad(T)$-invariant pluriclosed solution $g(t)$ therefore converges, as $t\to\infty$, to a bi-invariant metric (see Theorem \ref{PF-thm} for a more precise statement).  We note that the ODE systems \eqref{PF2} and \eqref{PF-intro} clearly describe the pluriclosed flow evolution among a large class of metrics.  This may for instance be used to establish the evolution behavior of many geometric quantities along these pluriclosed flow solutions.

\vspace{.3cm}\noindent 
{\it Acknowledgements}.  We thank Beatrice Brienza and Ramiro Lafuente for fruitful conversations and to Fabio Podest\`a for bringing the paper \cite{AlkDvd} to our attention.  We are specially grateful to Joseph Kwong for pointing out a mistake in the proof of Theorem 3.7 in a previous version of this paper, and for helping us to fix it.  We also thank the anonymous referee for Remark \ref{refe}.

\section{Hermitian geometry of Lie groups}\label{preli} 

Let $G$ be a connected and compact semisimple Lie group of even dimension.  The rank of $G$ is therefore also even, say $\rank(G)=2d$.  We refer to \S\ref{roots} for the definitions and notation concerning the root system of $G$.  

It was proved by Pittie \cite{Ptt} that any left-invariant complex structure on $G$ is one of the found independently by Samelson \cite{Sml} and Wang \cite{Wng}, which are defined as follows.  Given a maximal torus $T\subset G$ and an ordering $\Delta^+$ of roots, we consider $J:\ggo\rightarrow\ggo$ given by: the $\kil_\ggo$-orthogonal decomposition $\ggo=\tg\oplus\qg$ is $J$-invariant, $J_\tg:=J|_\tg$ is any linear map such that $J_\tg^2=-I$ and if
$$
\ggo^c=\tg^c\oplus\qg^c, \qquad \qg^c=\bigoplus_{\alpha\in\Delta} \CC E_\alpha, 
$$
then the complex linear map determined by $J_\qg:=J|_\qg$ on $\qg^c$ is given by
$$
J_\qg E_\alpha=\im E_\alpha, \quad\forall\alpha\in\Delta^+, \qquad 
J_\qg E_\alpha=-\im E_\alpha, \quad\forall\alpha\in\Delta^-:=-\Delta^+.  
$$
The biholomorphism class of the complex manifold $(G,J)$, $J=(J_\tg,J_\qg)$ does not depend on the choice of $T$ nor of $\Delta^+$, so we can fix $J_\qg$.  Two different $J_\tg$'s produce biholomorphic complex structures if and only if they belong to the same orbit of certain finite subgroup of $\Aut(G)$ (see \cite{Ptt}).  This implies that the space $\cca^G$ of all left-invariant complex structures on $G$ depends on $2d^2$ parameters up to biholomorphism (recall that $\dim{T}=2d$).  

The left and $\Ad(T)$-invariant metrics which are compatible with $J=(J_\tg,J_\qg)$ are all of the form $g=g_\tg+g_\qg$, where $g_\tg$ is any inner product on $\tg$ compatible with $J_\tg$ and $g_\qg=(x_\alpha)_{\alpha\in\Delta}$ is given by 
$$
g_\qg=\sum_{\alpha\in\Delta^+} x_\alpha (-\kil_\ggo)|_{\qg_\alpha}, \qquad 
\qg_\alpha:=(\CC E_\alpha\oplus\CC E_{-\alpha})\cap\ggo. 
$$
Equivalently,    
$$
\left\{\tfrac{1}{\sqrt{x_\alpha}}e_\alpha, \; \tfrac{1}{\sqrt{x_\alpha}}J_\qg e_\alpha : \alpha\in\Delta^+\right\} 
$$
is a $g$-orthonormal basis of $\qg$.  In particular, $x_\alpha>0$ and we set $x_{-\alpha}:=x_\alpha$ for any $\alpha\in\Delta^+$. The corresponding symmetric $\CC$-bilinear form on $\ggo^c$, also denoted by $g$, satisfies that 
$$
g(E_\alpha,E_{-\alpha})=-x_\alpha, \qquad\forall \alpha\in\Delta, 
$$ 
and $g(E_\alpha,E_\beta)=0$ whenever $\alpha+\beta\ne 0$.  The {\it Killing} metric $\gk:=-\kil_\ggo$ has $x_\alpha=1$ for any $\alpha\in\Delta^+$ and a bi-invariant metric $g_b=z_1(-\kil_{\ggo_1})+\dots+z_s(-\kil_{\ggo_s})$ corresponds to $x_\alpha=z_i$ if and only if $\alpha\in\Delta^+_i$, where $G=G_1\times\dots\times G_s$ is the decomposition in simple factors and $\Delta_i$ is the root system of $G_i$.  

We note that $(G,J)$ is {\it irreducible} as a complex manifold if and only if $J_\tg$ does not leave any subspace $\tg_{i_1}\oplus\dots\oplus\tg_{i_k}$ with $0<k<s$ invariant, where $\tg_i$ is the maximal torus of $\ggo_i$ (cf.\ \cite[Section 2]{Brb}).  In that case, there exists a unique bi-invariant metric compatible with $J$ (up to scaling), which is not necessarily the Killing metric as the following example shows.  

\begin{example}\label{G1G2}
Consider $G=G_1\times G_2$ such that $G_1,G_2$ are simple, $\rank(G_1)=\rank(G_2)$ and relative to the decomposition $\tg=\tg_1\oplus\tg_2$,   
$$
J_\tg=\left[\begin{matrix} 0&-\tfrac{1}{b}I\\ bI&0 \end{matrix}\right], \qquad b\ne 0.
$$
Then $J$ is irreducible and a bi-invariant metric $g_b=z_1(-\kil_{\ggo_1})+z_2(-\kil_{\ggo_s})$ is compatible with $J$ if and only if $\tfrac{z_1}{z_2}=b$.
\end{example}   

We consider the spaces $\mca^T$, $\cca^T$, $\hca^T_J$, $\cca^T_g$ and $\hca^T$ of all left and $\Ad(T)$-invariant metrics, complex structures, metrics compatible with a fixed $J\in\cca^T$, complex structures compatible with a fixed $g\in\mca^T$ and Hermitian structures, respectively.  They are all manifolds and it is easy to check that their dimensions are respectively given by  
$$
\begin{array}{c}
\dim{\mca^T}=\tfrac{d(d+1)}{2}+|\Delta^+|, \qquad 
\dim{\cca^T}=2d^2, \\ \\
\dim{\hca^T_J}=d^2+|\Delta^+|, \qquad 
\dim{\cca^T_g}=d(d-1), \qquad 
\dim{\hca^T}=3d^2+|\Delta^+|. 
\end{array}
$$
The K\"ahler form $\omega=g(J\cdot,\cdot)$ of a Hermitian structure $(J,g)$ on $G$ as above is given by 
$$
\omega=\omega_\tg+\omega_\qg, \qquad \mbox{where} \quad \omega_\tg=g_\tg(J_\tg\cdot,\cdot), \quad \omega_\qg(E_\alpha,E_{-\alpha})=-\im \epsilon_\alpha x_\alpha, \quad\forall \alpha\in\Delta,  
$$
where $\epsilon_\alpha:=\pm 1$ for $\alpha\in\Delta^\pm$.  Recall that if we consider $H_\alpha:=[E_\alpha,E_{-\alpha}]\in\tg^c$, then the set $\{ H_\alpha:\alpha\in\Delta^+\}$ generates $\tg^c$.     

\begin{lemma}\label{dco}
For any left and $\Ad(T)$-invariant Hermitian structure $(J,g)$ on a compact semisimple Lie group $G$, the only possibly nonzero components of the $3$-forms $d\omega$ and $d^c\omega:=-d\omega(J\cdot,J\cdot,J\cdot)$ are given by 
$$
d\omega(E_\alpha,E_\beta,E_\gamma)= N_{\alpha,\beta}(y_\alpha+y_\beta+y_\gamma), \qquad d^c\omega(E_\alpha,E_\beta,E_\gamma)= \im\epsilon_\alpha\epsilon_\beta\epsilon_\gamma N_{\alpha,\beta}(y_\alpha+y_\beta+y_\gamma), 
$$
whenever $\alpha+\beta+\gamma=0$, where $y_\alpha=-\im\epsilon_\alpha x_\alpha$ and $[E_\alpha,E_\beta]=N_{\alpha,\beta}E_{\alpha+\beta}$ (see \S\ref{roots}), and
$$
d\omega(A,E_\alpha,E_{-\alpha}) =\omega_\tg(A,H_\alpha)=g_\tg(J_\tg A,H_\alpha), \qquad d^c\omega(A,E_\alpha,E_{-\alpha}) =g_\tg(A,H_\alpha),
$$
for all $A\in\tg^c$.  
%On the other hand, 
%$$
%d\omega(A,e_\alpha,f_\alpha) = \omega_\ag(A,(\im H_\alpha)_\ag) 
%= g_\tg(J_\ag A, (\im H_\alpha)_\ag),
%$$
%where $(\cdot)_\ag$ denotes the projection on $\ag$ with respect to $\zg(\hg)=\zg(\kg)\oplus\ag$.  
\end{lemma}

\begin{remark}
In particular, $d\omega$ is always nonzero, in accordance with the well-known fact that $(J,g)$ is never K\"ahler.  It also follows that the Hermitian structure $(J_\qg,g_\qg)$ on the full flag manifold $G/T$ is K\"ahler if and only if $x_{\alpha+\beta}=x_\alpha+x_\beta$ for all $\alpha,\beta, \alpha+\beta\in\Delta^+$.  
\end{remark}

\begin{proof}
For any $\alpha,\beta,\gamma\in\Delta$, 
\begin{align*}
d\omega(E_\alpha,E_\beta,E_\gamma) &= 
-N_{\alpha,\beta}\omega_\qg(E_{\alpha+\beta},E_\gamma) 
+ N_{\alpha,\gamma}\omega_\qg(E_{\alpha+\gamma},E_\beta) 
- N_{\beta,\gamma}\omega_\qg(E_{\beta+\gamma},E_\alpha) \\ 
&= N_{\alpha,\beta}y_\gamma 
- N_{\alpha,\gamma}y_\beta 
+ N_{\beta,\gamma}y_\alpha 
= N_{\alpha,\beta}(y_\alpha+y_\beta+y_\gamma), 
\end{align*}
if $\alpha+\beta+\gamma=0$ and zero otherwise.  On the other hand, if $A\in\tg^c$, then
\begin{align*}
d\omega(A,E_\alpha,E_\beta) &= 
-\alpha(A)\omega_\qg(E_\alpha,E_\beta) 
+ \beta(A)\omega_\qg(E_\beta,E_\alpha) 
- \omega([E_\alpha,E_\beta],A) \\ 
&= \left\{\begin{array}{l} 
0, \qquad \alpha+\beta\ne 0, \\ 
\omega_\tg(A,H_\alpha), \qquad \alpha+\beta= 0.  
\end{array}\right.
\end{align*}
The only other possibilities, say $d\omega(A,B,E_\alpha)$ and $d\omega(A,B,C)$ for $B,C\in\tg^c$, are easily seen to vanish, concluding the proof.
\end{proof}

Left-invariant $1$-forms and closed $2$-forms on $G$ are all respectively given by 
$$
\theta_X:=g(\cdot,X),  \qquad \sigma_X:=\kil_\ggo([\cdot,\cdot],X), \quad X\in\ggo,
$$ 
and using that $\sigma_{PX}=d\theta_X$ for all $X$, where $g=\gk(P\cdot,\cdot)$, one obtains the well-known fact that the first and second Betti numbers of $G$ both vanish (see e.g.\ \cite[Section 2]{H3}).  
   
The following vectors in $\tg$ play a crucial role in the geometry of the Hermitian manifold $(G,J,g)$ (see e.g.\ \cite{AlkPrl,Grn,Ksz,Pds}): 
\begin{equation}\label{koszul-g}
Z_{J_\qg}:=\sum_{\alpha\in\Delta^+} \im H_\alpha, \qquad
Z_{J_\qg,g_\qg}:=\sum_{\alpha\in\Delta^+}\tfrac{1}{x_\alpha}  \im H_\alpha.  
\end{equation}
Note that $Z_{J_\qg}=Z_{J_\qg,\gk|_\qg}$.  

\begin{lemma}\label{dto} 
For any left and $\Ad(T)$-invariant Hermitian structure $(J,g)$ on a compact semisimple Lie group $G$,
$$
d_g^*\omega = -\theta_{Z_{J_\qg,g_\qg}} \quad\mbox{and}\quad 
dd_g^*\omega= -\sigma_{P_\tg Z_{J_\qg,g_\qg}}, 
\quad\mbox{where} \quad g_\tg=\gk(P_\tg\cdot,\cdot). 
$$
\end{lemma}

\begin{remark}
In particular, $d_g^*\omega =-g(\cdot,Z_{J_\qg,g_\qg})$ and 
$$
dd_g^*\omega= -\kil_\ggo([\cdot,\cdot],P_\tg Z_{J_\qg,g_\qg}) =g([\cdot,\cdot],Z_{J_\qg,g_\qg}).
$$  
\end{remark}

\begin{proof}
For any $X\in\ggo$, we have that 
\begin{align*}
d_g^*\omega(X) &= g(d_g^*\omega,\theta_X) = g(\omega,d\theta_X) 
%=\red\unm\black\sum\omega(X_i,X_j)d\theta_X(X_i,X_j) \\ 
%&= -\unm\sum\omega(X_i,X_j)\theta_X([X_i,X_j]_\pg) 
= -\sum_{\alpha\in\Delta_\qg^+} \tfrac{1}{x_\alpha^2}\omega(e_\alpha,J_\qg e_\alpha)\theta_X([e_\alpha,J_\qg e_\alpha]) \\ 
&= -\sum_{\alpha\in\Delta_\qg^+} \tfrac{1}{x_\alpha}\theta_X(\im H_\alpha) 
= -\theta_X(Z_{J_\qg,g_\qg}) = -g(X,Z_{J_\qg,g_\qg}), 
\end{align*}
concluding the proof.
\end{proof}

\section{Pluriclosed metrics}\label{SKT-sec} 

A Hermitian structure $(J,g)$ is called {\it pluriclosed} if $dd^c\omega=0$ (or equivalently, $\partial\overline{\partial}\omega=0$).  We start this section by computing the $4$-form $dd^c\omega$ in our setting.  
 
\begin{lemma}\label{ddco}
For any left and $\Ad(T)$-invariant Hermitian structure $(J,g)$ on a compact semisimple Lie group $G$, the only possibly nonzero components of the $4$-form $dd^c\omega$ are given by 
\begin{align*}
dd^c\omega(E_\alpha,E_{-\alpha},E_\beta,E_{-\beta})
=& -2g_\tg(H_\alpha,H_\beta) 
- 2N_{\alpha,\beta}^2(x_{\alpha+\beta}-x_{\alpha}-x_{\beta}) \\
& - 2\epsilon_{\alpha-\beta}N_{\alpha,-\beta}^2(\epsilon_{\alpha-\beta}x_{\alpha-\beta}-x_{\alpha}+x_\beta), \qquad \forall \alpha,\beta\in\Delta^+, \alpha\ne\beta,
\end{align*}
and if $\alpha+\beta+\gamma+\delta=0$ and all pairs add nonzero, then for all $\alpha,\beta\in\Delta^+$, $\gamma,\delta\in\Delta^-$,   
\begin{align*}
dd^c\omega(E_\alpha,E_\beta,E_\gamma,E_\delta) 
=& N_{\alpha,\beta} N_{\gamma,\delta} (x_\alpha+x_\beta+x_\gamma+x_\delta-2x_{\alpha+\beta}) \\ 
&- \epsilon_{\alpha+\gamma}N_{\alpha,\gamma} N_{\beta,\delta}(-x_\alpha+x_\beta+x_\gamma-x_\delta+2\epsilon_{\alpha+\gamma}x_{\alpha+\gamma}) \\
&+ \epsilon_{\alpha+\delta}N_{\alpha,\delta}  N_{\beta,\gamma} (-x_\alpha+x_\beta-x_\gamma+x_\delta+2\epsilon_{\alpha+\delta}x_{\alpha+\delta}).
\end{align*}
\end{lemma}

\begin{remark}
Using the properties of the numbers $N_{\alpha,\beta}$ given in \S\ref{roots}, one can easily check that $\gk$ is therefore always pluriclosed.  However, recall that there is a large class of left-invariant complex structures on $G$ which are not compatible with any bi-invariant metric.   
\end{remark}

\begin{remark}
These formulas are in accordance with those given in \cite{AlkDvd} for full flag manifolds.  
\end{remark}

\begin{proof}
We use the formula for $d^c\omega$ given in Lemma \ref{dco}.  If $\alpha+\beta+\gamma=0$, then for any $A\in\tg^c$, 
\begin{align*}
& dd^c\omega(A,E_\alpha,E_\beta,E_\gamma) \\ 
=& -\alpha(A)d^c\omega(E_\alpha,E_\beta,E_\gamma) 
+\beta(A)d^c\omega(E_\beta,E_\alpha,E_\gamma)
-\gamma(A)d^c\omega(E_\gamma,E_\alpha,E_\beta)\\ 
&-N_{\alpha,\beta}d^c\omega(E_{\alpha+\beta},A,E_\gamma) 
+ N_{\alpha,\gamma}d^c\omega(E_{\alpha+\gamma},A,E_\beta) 
- N_{\beta,\gamma}d^c\omega(E_{\beta+\gamma},A,E_\alpha) \\ 
%=& -(\alpha+\beta+\gamma)(A)d^c\omega(E_\alpha,E_\beta,E_\gamma) \\  
%&-N_{\alpha,\beta}d^c\omega(E_{\alpha+\beta},A,E_\gamma) 
%+ N_{\alpha,\gamma}d^c\omega(E_{\alpha+\gamma},A,E_\beta) 
%- N_{\beta,\gamma}d^c\omega(E_{\beta+\gamma},A,E_\alpha) \\ 
=& -(\alpha+\beta+\gamma)(A)d^c\omega(E_\alpha,E_\beta,E_\gamma) \\  
&-N_{\alpha,\beta}d^c\omega(A,E_\gamma,E_{\alpha+\beta}) 
+ N_{\alpha,\gamma}d^c\omega(A,E_\beta,E_{\alpha+\gamma}) 
- N_{\beta,\gamma}d^c\omega(A,E_\alpha,E_{\beta+\gamma}) \\ 
%=& -(\alpha+\beta+\gamma)(A)d^c\omega(E_\alpha,E_\beta,E_\gamma) \\  
=&-N_{\alpha,\beta}g_\tg(A,H_\gamma) 
+ N_{\alpha,\gamma}g_\tg(A,H_\beta) 
- N_{\beta,\gamma}g_\tg(A,H_\alpha) \\
=&-N_{\alpha,\beta}g_\tg(A,H_\alpha+H_\beta+H_\gamma) 
=-N_{\alpha,\beta}g_\tg(A,H_{\alpha+\beta+\gamma})=0. 
\end{align*}
For $A,B\in\tg^c$ and $\alpha+\beta=0$, 
\begin{align*}
dd^c\omega(A,B,E_\alpha,E_\beta)  
=& \alpha(A)d^c\omega(E_\alpha,B,E_\beta) 
-\beta(A)d^c\omega(E_\beta,B,E_\alpha)
-\alpha(B)d^c\omega(E_\alpha,A,E_\beta) \\ 
&+\beta(B)d^c\omega(E_\beta,A,E_\alpha)
-N_{\alpha,\beta}d^c\omega(E_{\alpha+\beta},A,B) \\ 
=&-(\alpha+\beta)(A)d^c\omega(B,E_\alpha,E_\beta) 
+(\alpha+\beta)(A)d^c\omega(A,E_\alpha,E_\beta) =0.
\end{align*}
If $\alpha+\beta+\gamma+\delta=0$ and all pairs add nonzero, then
\begin{align*}
& dd^c\omega(E_\alpha,E_\beta,E_\gamma,E_\delta) \\ 
%=& -N_{\alpha,\beta}d^c\omega(E_{\alpha+\beta},E_\gamma,E_\delta) 
%+ N_{\alpha,\gamma}d^c\omega(E_{\alpha+\gamma},E_\beta,E_\delta) 
%- N_{\alpha,\delta}d^c\omega(E_{\alpha+\delta},E_\beta,E_\gamma) \\
%& - N_{\beta,\gamma}d^c\omega(E_{\beta+\gamma},E_\alpha,E_\delta) 
%+ N_{\beta,\delta}d^c\omega(E_{\beta+\delta},E_\alpha,E_\gamma)
%- N_{\gamma,\delta}d^c\omega(E_{\gamma+\delta},E_\alpha,E_\beta) \\ 
%
=& -N_{\alpha,\beta}d^c\omega(E_{\alpha+\beta},E_\gamma,E_\delta) 
- N_{\gamma,\delta}d^c\omega(E_{\gamma+\delta},E_\alpha,E_\beta) \\ 
&+ N_{\alpha,\gamma}d^c\omega(E_{\alpha+\gamma},E_\beta,E_\delta) 
+ N_{\beta,\delta}d^c\omega(E_{\beta+\delta},E_\alpha,E_\gamma) \\
&- N_{\alpha,\delta}d^c\omega(E_{\alpha+\delta},E_\beta,E_\gamma) 
- N_{\beta,\gamma}d^c\omega(E_{\beta+\gamma},E_\alpha,E_\delta) \\
=& -N_{\alpha,\beta}\im \epsilon_{\alpha+\beta}\epsilon_\gamma\epsilon_\delta N_{\gamma,\delta} (y_{\alpha+\beta}+y_\gamma+y_\delta) 
- N_{\gamma,\delta}\im \epsilon_{\gamma+\delta}\epsilon_\alpha\epsilon_\beta N_{\alpha,\beta} (y_{\gamma+\delta}+y_\alpha+y_\beta) \\ 
&+ N_{\alpha,\gamma}\im \epsilon_{\alpha+\gamma}\epsilon_\beta\epsilon_\delta N_{\beta,\delta}(y_{\alpha+\gamma}+y_\beta+y_\delta)  
+ N_{\beta,\delta} \im\epsilon_{\beta+\delta}\epsilon_\alpha\epsilon_\gamma N_{\alpha,\gamma} (y_{\beta+\delta}+y_\alpha+y_\gamma) \\
&- N_{\alpha,\delta}\im \epsilon_{\alpha+\delta}\epsilon_\beta\epsilon_\gamma N_{\beta,\gamma} (y_{\alpha+\delta}+y_\beta+y_\gamma) 
- N_{\beta,\gamma}\im \epsilon_{\beta+\gamma}\epsilon_\alpha\epsilon_\delta N_{\alpha,\delta}(y_{\beta+\gamma}+y_\alpha+y_\delta). 
\end{align*}
In particular, for $\alpha,\beta,\gamma\in\Delta^+$, $\delta\in\Delta^-$, we obtain
\begin{align*}
& dd^c\omega(E_\alpha,E_\beta,E_\gamma,E_\delta) \\ 
%=& -N_{\alpha,\beta}\im \epsilon_{\alpha+\beta}\epsilon_\gamma\epsilon_\delta N_{\gamma,\delta} (y_{\alpha+\beta}+y_\gamma+y_\delta) 
%- N_{\gamma,\delta}\im \epsilon_{\gamma+\delta}\epsilon_\alpha\epsilon_\beta N_{\alpha,\beta} (y_{\gamma+\delta}+y_\alpha+y_\beta) \\ 
%&+ N_{\alpha,\gamma}\im \epsilon_{\alpha+\gamma}\epsilon_\beta\epsilon_\delta N_{\beta,\delta}(y_{\alpha+\gamma}+y_\beta+y_\delta)  
%+ N_{\beta,\delta} \im\epsilon_{\beta+\delta}\epsilon_\alpha\epsilon_\gamma N_{\alpha,\gamma} (y_{\beta+\delta}+y_\alpha+y_\gamma) \\
%&- N_{\alpha,\delta}\im \epsilon_{\alpha+\delta}\epsilon_\beta\epsilon_\gamma N_{\beta,\gamma} (y_{\alpha+\delta}+y_\beta+y_\gamma) 
%- N_{\beta,\gamma}\im \epsilon_{\beta+\gamma}\epsilon_\alpha\epsilon_\delta N_{\alpha,\delta}(y_{\beta+\gamma}+y_\alpha+y_\delta). \\
%
%=& N_{\alpha,\beta} N_{\gamma,\delta} (x_{\alpha+\beta}+x_\gamma-x_\delta) 
%- \epsilon_{\gamma+\delta}N_{\gamma,\delta} N_{\alpha,\beta} (\epsilon_{\gamma+\delta}x_{\gamma+\delta}+x_\alpha+x_\beta) \\ 
%& - N_{\alpha,\gamma}  N_{\beta,\delta}(x_{\alpha+\gamma}+x_\beta-x_\delta)  
%+ \epsilon_{\beta+\delta}N_{\beta,\delta}  N_{\alpha,\gamma} (\epsilon_{\beta+\delta}x_{\beta+\delta}+x_\alpha+x_\gamma) \\
%&-  \epsilon_{\alpha+\delta}N_{\alpha,\delta} N_{\beta,\gamma} (\epsilon_{\alpha+\delta}x_{\alpha+\delta}+x_\beta+x_\gamma) 
%+ N_{\beta,\gamma} N_{\alpha,\delta}(x_{\beta+\gamma}+x_\alpha-x_\delta). \\ 
%
=& N_{\alpha,\beta} N_{\gamma,\delta} (x_{\alpha+\beta}+x_\gamma-x_\delta) 
+N_{\gamma,\delta} N_{\alpha,\beta} (-x_{\gamma+\delta}+x_\alpha+x_\beta) \\ 
& - N_{\alpha,\gamma}  N_{\beta,\delta}(x_{\alpha+\gamma}+x_\beta-x_\delta)  
-N_{\beta,\delta}  N_{\alpha,\gamma} (-x_{\beta+\delta}+x_\alpha+x_\gamma) \\
&+N_{\alpha,\delta} N_{\beta,\gamma} (-x_{\alpha+\delta}+x_\beta+x_\gamma) 
+ N_{\beta,\gamma} N_{\alpha,\delta}(x_{\beta+\gamma}+x_\alpha-x_\delta) \\ 
%
%=& N_{\alpha,\beta} N_{\gamma,\delta} (x_\alpha+x_\beta+x_\gamma-x_\delta) \\ 
%& - N_{\alpha,\gamma}  N_{\beta,\delta}(x_\alpha+x_\beta+x_\gamma-x_\delta) \\
%&+N_{\alpha,\delta} N_{\beta,\gamma} (x_\alpha+x_\beta+x_\gamma-x_\delta) \\
%
=& (N_{\alpha,\beta} N_{\gamma,\delta} - N_{\alpha,\gamma}N_{\beta,\delta}+N_{\alpha,\delta} N_{\beta,\gamma})(x_\alpha+x_\beta+x_\gamma-x_\delta) =0.  
\end{align*}
Note that the above computation can be skipped by using that $dd^c\omega=-\partial\overline{\partial}\omega$ is a $(2,2)$-form.  On the other hand, for $\alpha,\beta\in\Delta^+$, $\gamma,\delta\in\Delta^-$, we have that  
\begin{align*}
& dd^c\omega(E_\alpha,E_\beta,E_\gamma,E_\delta) \\ 
%=& -N_{\alpha,\beta}\im \epsilon_{\alpha+\beta}\epsilon_\gamma\epsilon_\delta N_{\gamma,\delta} (y_{\alpha+\beta}+y_\gamma+y_\delta) 
%- N_{\gamma,\delta}\im \epsilon_{\gamma+\delta}\epsilon_\alpha\epsilon_\beta N_{\alpha,\beta} (y_{\gamma+\delta}+y_\alpha+y_\beta) \\ 
%&+ N_{\alpha,\gamma}\im \epsilon_{\alpha+\gamma}\epsilon_\beta\epsilon_\delta N_{\beta,\delta}(y_{\alpha+\gamma}+y_\beta+y_\delta)  
%+ N_{\beta,\delta} \im\epsilon_{\beta+\delta}\epsilon_\alpha\epsilon_\gamma N_{\alpha,\gamma} (y_{\beta+\delta}+y_\alpha+y_\gamma) \\
%&- N_{\alpha,\delta}\im \epsilon_{\alpha+\delta}\epsilon_\beta\epsilon_\gamma N_{\beta,\gamma} (y_{\alpha+\delta}+y_\beta+y_\gamma) 
%- N_{\beta,\gamma}\im \epsilon_{\beta+\gamma}\epsilon_\alpha\epsilon_\delta N_{\alpha,\delta}(y_{\beta+\gamma}+y_\alpha+y_\delta). \\
%
=& -N_{\alpha,\beta} N_{\gamma,\delta} (x_{\alpha+\beta}-x_\gamma-x_\delta) 
+ N_{\gamma,\delta} N_{\alpha,\beta} (-x_{\gamma+\delta}+x_\alpha+x_\beta) \\ 
&- \epsilon_{\alpha+\gamma}N_{\alpha,\gamma} N_{\beta,\delta}(\epsilon_{\alpha+\gamma}x_{\alpha+\gamma}+x_\beta-x_\delta)  
- \epsilon_{\beta+\delta}N_{\beta,\delta}  N_{\alpha,\gamma} (\epsilon_{\beta+\delta}x_{\beta+\delta}+x_\alpha-x_\gamma) \\
&+ \epsilon_{\alpha+\delta}N_{\alpha,\delta}  N_{\beta,\gamma} (\epsilon_{\alpha+\delta}x_{\alpha+\delta}+x_\beta-x_\gamma) 
+ \epsilon_{\beta+\gamma}N_{\beta,\gamma}  N_{\alpha,\delta}(\epsilon_{\beta+\gamma}x_{\beta+\gamma}+x_\alpha-x_\delta),   \\ 
%
%=& N_{\alpha,\beta} N_{\gamma,\delta} (x_\alpha+x_\beta+x_\gamma+x_\delta-2x_{\alpha+\beta}) \\ 
%&- \epsilon_{\alpha+\gamma}N_{\alpha,\gamma} N_{\beta,\delta}(\epsilon_{\alpha+\gamma}x_{\alpha+\gamma}+x_\beta-x_\delta-\epsilon_{\beta+\delta}x_{\beta+\delta}-x_\alpha+x_\gamma) \\
%&+ \epsilon_{\alpha+\delta}N_{\alpha,\delta}  N_{\beta,\gamma} (\epsilon_{\alpha+\delta}x_{\alpha+\delta}+x_\beta-x_\gamma-\epsilon_{\beta+\gamma}x_{\beta+\gamma}-x_\alpha+x_\delta), \\
%%
%=& N_{\alpha,\beta} N_{\gamma,\delta} (x_\alpha+x_\beta+x_\gamma+x_\delta-2x_{\alpha+\beta}) \\ 
%&- \epsilon_{\alpha+\gamma}N_{\alpha,\gamma} N_{\beta,\delta}(-x_\alpha+x_\beta+x_\gamma-x_\delta+\epsilon_{\alpha+\gamma}(x_{\alpha+\gamma}+x_{\beta+\delta})) \\
%&+ \epsilon_{\alpha+\delta}N_{\alpha,\delta}  N_{\beta,\gamma} (-x_\alpha+x_\beta-x_\gamma+x_\delta+\epsilon_{\alpha+\delta}(x_{\alpha+\delta}+x_{\beta+\gamma})), \\
%
=& N_{\alpha,\beta} N_{\gamma,\delta} (x_\alpha+x_\beta+x_\gamma+x_\delta-2x_{\alpha+\beta}) \\ 
&- \epsilon_{\alpha+\gamma}N_{\alpha,\gamma} N_{\beta,\delta}(-x_\alpha+x_\beta+x_\gamma-x_\delta+2\epsilon_{\alpha+\gamma}x_{\alpha+\gamma}) \\
&+ \epsilon_{\alpha+\delta}N_{\alpha,\delta}  N_{\beta,\gamma} (-x_\alpha+x_\beta-x_\gamma+x_\delta+2\epsilon_{\alpha+\delta}x_{\alpha+\delta}).
\end{align*}
Finally, if $\alpha,\beta\in\Delta^+$, $\alpha\pm\beta\ne 0$ then
\begin{align*}
& dd^c\omega(E_\alpha,E_{-\alpha},E_\beta,E_{-\beta}) \\ 
=& -d^c\omega(H_\alpha,E_\beta,E_{-\beta}) 
+ N_{\alpha,\beta}d^c\omega(E_{\alpha+\beta},E_{-\alpha},E_{-\beta}) 
- N_{\alpha,-\beta}d^c\omega(E_{\alpha-\beta},E_{-\alpha},E_\beta) \\
& - N_{-\alpha,\beta}d^c\omega(E_{-\alpha+\beta},E_{\alpha},E_{-\beta})
+ N_{-\alpha,-\beta}d^c\omega(E_{-\alpha-\beta},E_\alpha,E_\beta)
-d^c\omega(H_\beta,E_\alpha,E_{-\alpha}) \\ 
=& -g_\tg(H_\alpha,H_\beta) 
+ N_{\alpha,\beta}\im\epsilon_{\alpha+\beta}\epsilon_{-\alpha}\epsilon_{-\beta} N_{\alpha+\beta,-\alpha}(y_{\alpha+\beta}+y_{-\alpha}+y_{-\beta}) \\
& - N_{\alpha,-\beta}\im\epsilon_{\alpha-\beta}\epsilon_{-\alpha}\epsilon_\beta N_{\alpha-\beta,-\alpha}(y_{\alpha-\beta}+y_{-\alpha}+y_\beta)\\ 
&- N_{-\alpha,\beta}\im\epsilon_{-\alpha+\beta}\epsilon_{\alpha}\epsilon_{-\beta} N_{-\alpha+\beta,\alpha}(y_{-\alpha+\beta}+y_{\alpha}+y_{-\beta})\\
&+ N_{-\alpha,-\beta}\im\epsilon_{-\alpha-\beta}\epsilon_\alpha\epsilon_\beta N_{-\alpha-\beta,\alpha}(y_{-\alpha-\beta}+y_\alpha+y_\beta)
-g_\tg(H_\beta,H_\alpha), \\
=& -2g_\tg(H_\alpha,H_\beta) 
+ 2N_{\alpha,\beta}\im N_{\alpha+\beta,-\alpha}(y_{\alpha+\beta}+y_{-\alpha}+y_{-\beta}) \\
& + 2N_{\alpha,-\beta}\im\epsilon_{\alpha-\beta} N_{\alpha-\beta,-\alpha}(y_{\alpha-\beta}+y_{-\alpha}+y_\beta)\\
=& -2g_\tg(H_\alpha,H_\beta) 
- 2N_{\alpha,\beta}^2 (x_{\alpha+\beta}-x_{\alpha}-x_{\beta}) 
- 2N_{\alpha,-\beta}^2\epsilon_{\alpha-\beta} (\epsilon_{\alpha-\beta}x_{\alpha-\beta}-x_{\alpha}+x_\beta),
\end{align*}
concluding the proof. 
\end{proof}

Recall that the set $\{ \im H_\alpha:\alpha\in\Delta^+\}$ generates $\tg$.  

\begin{corollary}\label{SKT}
A left and $\Ad(T)$-invariant Hermitian structure $(J,g)$ on a compact semisimple Lie group $G$ is pluriclosed if and only if for all $\alpha,\beta\in\Delta^+$, $\alpha\ne\beta$, 
\begin{align}
g_\tg(\im H_\alpha,\im H_\beta)= 
N_{\alpha,\beta}^2(x_{\alpha+\beta}-x_{\alpha}-x_{\beta}) 
+\epsilon_{\alpha-\beta}N_{\alpha,-\beta}^2(\epsilon_{\alpha-\beta}x_{\alpha-\beta}-x_{\alpha}+x_\beta), \label{skt1} 
\end{align}
and for all $\alpha,\beta\in\Delta^+$, $\gamma,\delta\in\Delta^-$ such that $\alpha+\beta+\gamma+\delta=0$ and all pairs add nonzero,
\begin{align}
0 =& N_{\alpha,\beta} N_{\gamma,\delta} (x_\alpha+x_\beta+x_\gamma+x_\delta-2x_{\alpha+\beta}) \notag \\ 
&- \epsilon_{\alpha+\gamma}N_{\alpha,\gamma} N_{\beta,\delta}(-x_\alpha+x_\beta+x_\gamma-x_\delta+2\epsilon_{\alpha+\gamma}x_{\alpha+\gamma}) \label{skt2}\\
&+ \epsilon_{\alpha+\delta}N_{\alpha,\delta}  N_{\beta,\gamma} (-x_\alpha+x_\beta-x_\gamma+x_\delta+2\epsilon_{\alpha+\delta}x_{\alpha+\delta}). \notag
\end{align}
\end{corollary}

\begin{remark}
This characterization coincides with one obtained in the case of a non-compact simple Lie group of inner type in \cite[(4.2)]{GstPds}.  
\end{remark}

As a first application of the above characterization of pluriclosed metrics, actually of only condition \eqref{skt1}, we provide an alternative proof for \cite[Theorem 3.1]{FinGrn1}.  

\begin{proposition}\label{SKT3}\cite{FinGrn1}
For any left and $\Ad(T)$-invariant pluriclosed Hermitian structure $(J,g)$ on a compact semisimple Lie group $G$, one has that  $g_\tg=g_b|_\tg$ for some bi-invariant metric $g_b$.
\end{proposition}

\begin{proof}
Recall that $\{\im H_\alpha:\alpha\in\Pi\}$ is a basis of $\tg$, where $\Pi\subset\Delta^+$ is the set of simple roots.  We set 
$$
\la\alpha,\beta\ra:=\gk(\im H_\alpha,\im H_\beta) =\kil(H_\alpha,H_\beta), \qquad 
g_\tg(\alpha,\beta):=g_\tg(\im H_\alpha,\im H_\beta), \quad \forall \alpha,\beta\in\Delta. 
$$  
For any $\alpha,\beta\in\Pi$ such that $\la\alpha,\beta\ra=0$, we obtain from \eqref{skt1} that $g_\tg(\alpha,\beta)=0=\la\alpha,\beta\ra$ by using that $\alpha\pm\beta\notin\Delta$.  In particular, the maximal tori of two different simple factors are $g_\tg$-orthogonal, so we can assume that $G$ is simple and show that $g_\tg=\gk|_\tg$ up to scaling.  

On the other hand, if $\alpha,\beta\in\Pi$ satisfy that $\la\alpha,\beta\ra<0$, then there are exactly three possibilities (up to scaling we can assume that $g_\tg(\alpha,\alpha)=\la\alpha,\alpha\ra$): 
\begin{enumerate}[{\rm (a)}]
\item $\dynkin[scale=2, labels={\alpha,\beta}] A{oo}\;$: $\;\la\alpha,\alpha\ra=\la\beta,\beta\ra=-2\la\alpha,\beta\ra$, $N_{\alpha,\beta}^2=\unm\la\alpha,\alpha\ra=N_{\alpha+\beta,\alpha}^2=N_{\alpha+\beta,\beta}^2$.  It follows from \eqref{skt1} that 
\begin{align*}
g_\tg(\alpha,\beta)=&\unm\la\alpha,\alpha\ra(x_{\alpha+\beta}-x_\alpha-x_\beta), \\
g_\tg(\alpha+\beta,\alpha)=&\unm\la\alpha,\alpha\ra(x_\beta-x_{\alpha+\beta}+x_\alpha), \\
g_\tg(\alpha+\beta,\beta)=&\unm\la\alpha,\alpha\ra(x_\alpha-x_{\alpha+\beta}+x_\beta),   
\end{align*}
which implies that $g_\tg(\alpha,\beta)=\la\alpha,\beta\ra$, $g_\tg(\beta,\beta)=\la\beta,\beta\ra$.  

\item $\dynkin[scale=2, labels={\alpha,\beta}] C{oo}\;$: $\;\la\alpha,\alpha\ra=\unm\la\beta,\beta\ra=-\la\alpha,\beta\ra$, $N_{\alpha,\beta}^2=\la\alpha,\alpha\ra=N_{\alpha+\beta,\alpha}^2=N_{2\alpha+\beta,\alpha+\beta}^2$.  By \eqref{skt1} we have that $g_\tg(2\alpha+\beta,\beta)=0$, so $g_\tg(\alpha,\beta)=-\unm g_\tg(\beta, \beta)$, and in addition,
\begin{align*}
g_\tg(\alpha,\beta)=&\la\alpha,\alpha\ra(x_{\alpha+\beta}-x_\alpha-x_\beta), \\
g_\tg(\alpha+\beta,\alpha)=&\la\alpha,\alpha\ra(x_{2\alpha+\beta}-x_{\alpha+\beta}-x_\alpha) 
+\la\alpha,\alpha\ra(x_\beta-x_{\alpha+\beta}+x_\alpha), \\
g_\tg(2\alpha+\beta,\alpha+\beta)=&\la\alpha,\alpha\ra(x_\alpha-x_{2\alpha+\beta}+x_{\alpha+\beta}).    
\end{align*}
It is easy to see that the only solutions are $g_\tg(\alpha,\beta)=\la\alpha,\beta\ra$, $g_\tg(\beta,\beta)=\la\beta,\beta\ra$. 

\item $\dynkin[scale=2, labels={\beta, \alpha}, backwards=true] G{oo}\;$: $\;\la\alpha,\alpha\ra=\tfrac{1}{3}\la\beta,\beta\ra=-\tfrac{2}{3}\la\alpha,\beta\ra$, $N_{\alpha,\beta}^2=\tfrac{3}{2}\la\alpha,\alpha\ra =N_{2\alpha+\beta,\alpha}^2 =N_{3\alpha+\beta,\beta}^2$, $N_{\alpha+\beta,\alpha}^2=2\la\alpha,\alpha\ra$.  Here condition \eqref{skt1} gives that $g_\tg(2\alpha+\beta,\beta)=0$ and $g_\tg(3\alpha+\beta,\alpha+\beta)=0$, from which clearly follows that $g_\tg(\alpha,\beta)=\la\alpha,\beta\ra$ and $g_\tg(\beta,\beta)=\la\beta,\beta\ra$.  
\end{enumerate}
All this allows us to proceed as follows.  We start at any end $\alpha_1$ of the Dynkin diagram and assume up to scaling that $g_\tg(\alpha_1,\alpha_1)=\la\alpha_1,\alpha_1\ra$.  Then, by using (a), (b) and (c) above, we run through the diagram obtaining in each step that 
$g_\tg|_{\spann\{\alpha_i,\alpha_{i+1}\}}=\ip|_{\spann\{\alpha_{i},\alpha_{i+1}\}}$ and the same for the next two nodes in any bifurcation.  Thus $g_\tg(\im H_\alpha,\im H_\beta)=\gk(\im H_\alpha,\im H_\beta)$ for any $\alpha,\beta\in\Pi$, that is, $g_\tg=\gk|_\tg$, as was to be shown.  
\end{proof}

We are now ready to prove the main result of this paper, a complete classification of left and $\Ad(T)$-invariant pluriclosed metrics on $G$.  If $G=G_1\times\dots\times G_s$, where $G_i$ is simple for all $i$, then for each $1\leq i\leq s$ we consider the decomposition $\ggo_i=\tg_i\oplus\qg_i$ such that $\tg=\tg_1\oplus\dots\oplus\tg_s$ and $\qg=\qg_1\oplus\dots\oplus\qg_s$, so each $\tg_i$ is a maximal abelian subalgebra of $\ggo_i$ and defines a root system $\Delta_i$ with $\Delta=\Delta_1\sqcup\dots\sqcup\Delta_s$.  

\begin{theorem}\label{SKT2}
On any compact semisimple Lie group $G$, the following conditions on a left and $\Ad(T)$-invariant Hermitian structure $(J,g)$ are equivalent: 
\begin{enumerate}[{\rm (i)}] 
\item The Hermitian manifold $(G,J,g)$, $g=g_\tg+g_\qg$ is pluriclosed. 

\item $g_\tg(\tg_i,\tg_j)=0$ for all $i\ne j$ and on each simple factor, up to scaling, $g_\tg=\gk|_\tg$ and $x_{\alpha+\beta}=x_\alpha+x_\beta-1$ for all $\alpha,\beta,\alpha+\beta\in\Delta^+$, where $g_\qg=(x_\alpha)_{\alpha\in\Delta^+}$.  

\item $g_\tg(\tg_i,\tg_j)=0$ for all $i\ne j$ and for each $1\leq i\leq s$, there exists $z_i>0$ such that $g_\tg|_{\tg_i}=z_i\gk|_{\tg_i}$ and $g_\qg|_{\qg_i}=(z_ix_\alpha)_{\alpha\in\Delta^+_i}$, where the $x_\alpha$ are defined as follows: if $\Pi_i=\{\alpha_1,\dots,\alpha_{n_i}\}$ are the simple roots of $\ggo_i$ and we set $x_j:=x_{\alpha_j}$ for $j=1,\dots,n_i$, then 
$$
x_\alpha:=\sum_{j=1}^{n_i} k^\alpha_jx_j -\left(\sum_{j=1}^{n_i} k^\alpha_j -1\right), \qquad \forall \alpha=\sum_{j=1}^{n_i} k^\alpha_j\alpha_j\in\Delta^+_i, \quad k^\alpha_j\in\ZZ_{\geq 0}.  
$$  
\end{enumerate}
\end{theorem}

\begin{remark}
Given any of these metrics $g$, the space of all left-invariant complex structures which are compatible with $g$ depends on $d(d-1)$ parameters, where $2d=\rank(G)$.  On the other hand, it follows from the theorem that for any fixed irreducible $J$, the space of all pluriclosed metrics compatible with $J$, if non-empty, depends on $2d+1$ parameters and contains a unique bi-invariant metric up to scaling, given by $g_b=z_1(-\kil_{\ggo_1})+\dots+z_s(-\kil_{\ggo_s})$.  Note that $g_\tg=g_b|_\tg$ for any metric $g$ as in the theorem.  In general, the number of parameters is $r(2d+1)$ and $J$ admits an $r$-parametric space of Hermitian bi-invariant metrics, where $r$ is the number of irreducible factors of $(G,J)$ as a complex manifold.  
\end{remark}

\begin{remark}\label{xpos}
The natural number $h(\alpha):=\sum\limits_{i=1}^n k^\alpha_i$ is called the {\it height} of the root $\alpha=\sum\limits_{i=1}^n k^\alpha_i\alpha_i$, where $\Pi=\{\alpha_1,\dots,\alpha_n\}$.  In order to obtain a metric it is of course necessary that $x_\alpha>0$ for all $\alpha\in\Delta^+$, which for instance holds if $x_i\geq 1-\tfrac{1}{h(\alpha_{max})}$ for all $i=1,\dots,n$, where $\alpha_{max}$ is the maximal root.  
\end{remark}

\begin{proof}
We first prove that (i) implies (ii) 
%Note that in the proof of Proposition \ref{SKT3}, we also obtain assuming that $g_\tg=\gk|_\tg$ the following: $x_{\alpha+\beta}=x_\alpha+x_\beta-1$ in cases (a) and (b), and $x_{2\alpha+\beta}=2x_\alpha+x_\beta-2$ in case (b).  In case (c), condition \eqref{skt1} gives the following equations when applied to the pairs $(\alpha,\beta)$, $(\alpha+\beta,\alpha)$, $(2\alpha+\beta,\alpha)$ and $(3\alpha+\beta,\beta)$, respectively:
%\begin{align*}
%-1=&x_{\alpha+\beta}-x_\alpha-x_\beta, \\
%-1=&4x_{2\alpha+\beta}-7x_{\alpha+\beta}+3x_\beta-x_\alpha, \\  
%1=&5x_{3\alpha+\beta}-11x_{2\alpha+\beta}+6x_{\alpha+\beta}+x_\alpha, \\ 
%-1=& x_{3\alpha+2\beta}-x_{3\alpha+\beta}-x_\beta.
%\end{align*}
%This implies that $x_{\alpha+\beta}=x_\alpha+x_\beta-1$, $x_{2\alpha+\beta}=2x_\alpha+x_\beta-2$, $x_{3\alpha+\beta}=3x_\alpha+x_\beta-3$ and $x_{3\alpha+2\beta}=3x_\alpha+2x_\beta-4$.  Summarizing, the formula in part (ii) holds for any two positive roots in $\spann\{\alpha,\beta\}$, where $\alpha,\beta\in\Pi$.  
by induction in the height of the roots. 
%A straightforward inspection of the nine types of Dynkin diagrams gives that for any $\gamma=\sum\limits_{i=1}^{2d}k_i\alpha_i\in\Delta^+$ such that at least three different integers $k_i$'s are $>0$, there exists $\alpha,\beta\in\Delta^+$ such that $\gamma=\alpha+\beta$ and $\alpha-\beta\notin\Delta$.  Thus $N_{\alpha,\beta}^2=-\la\alpha,\beta\ra$ and so by \eqref{skt1},  
%%$$
%%\la\alpha,\beta\ra = g_\tg(\alpha,\beta) =-\la\alpha,\beta\ra(x_{\alpha+\beta}-x_\alpha-x_\beta), 
%%$$
%$x_\gamma=x_{\alpha+\beta}=x_\alpha+x_\beta-1$.  If $\alpha=\sum\limits_{i=1}^{2d}l_i\alpha_i$ and $\beta=\sum\limits_{i=1}^{2d}m_i\alpha_i$ (so $k_i=l_i+m_i$), then by induction hypothesis, 
%$$
%x_\gamma=\sum_{i=1}^{2d}l_ix_i -(h(\alpha)-1) + \sum_{i=1}^{2d}m_ix_i -(h(\beta)-1) - 1 
%%=\sum_{i=1}^{2d}(l_i+m_i)x_i -(\sum_{i=1}^{2d}(l_i+m_i)-1) 
%=\sum_{i=1}^{2d}k_ix_i -(h(\gamma)-1), 
%$$
%as was to be shown.  
On any of the simple factors of $G$, we can assume by Proposition \ref{SKT3} that $g_\tg(\im H_\alpha,\im H_\beta)=\la\alpha,\beta\ra$ for all $\alpha,\beta\in\Delta$, so it follows from \eqref{skt1} that   
\begin{align*}
\la\alpha,\beta\ra= 
N_{\alpha,\beta}^2(x_{\alpha+\beta}-x_{\alpha}-x_{\beta}) 
+\epsilon_{\alpha-\beta}N_{\alpha,-\beta}^2(\epsilon_{\alpha-\beta}x_{\alpha-\beta}-x_{\alpha}+x_\beta), 
\end{align*}
for all $\alpha,\beta\in\Delta^+$, $\alpha\ne\beta$.  If $\epsilon_{\alpha-\beta}=0$, then $N_{\alpha,\beta}^2=-\la\alpha,\beta\ra$ and we are done.  Otherwise, we can assume that $\epsilon_{\alpha-\beta}=1$, i.e., $\alpha-\beta\in\Delta^+$, and since $x_{\alpha}=x_{\alpha-\beta}+x_\beta-1$ by induction hypothesis, we obtain that (see \S\ref{roots})  
\begin{align*}
\la\alpha,\beta\ra= 
N_{\alpha,\beta}^2(x_{\alpha+\beta}-x_{\alpha}-x_{\beta}) 
+N_{\alpha,\beta}^2+\la\alpha,\beta\ra, 
\end{align*}
so $x_{\alpha+\beta}-x_{\alpha}-x_{\beta}=-1$.  

On the other hand, it is easy to check that part (iii) is equivalent to part (ii).  

Finally, if part (ii) holds, then we also have that $x_{\alpha-\beta}-1=\pm(x_\alpha-x_\beta)$ if $\alpha-\beta\in\pm\Delta^+$ for all $\alpha,\beta\in\Delta^+$.  Let us prove part (i) by using Corollary \ref{SKT}.  We can assume that $G$ is simple and that $g_\tg=\gk|_\tg$, i.e., $g_\tg(\im H_\alpha,\im H_\beta)=\la\alpha,\beta\ra$ for all $\alpha,\beta\in\Delta^+$.  By using that $N_{\alpha,-\beta}^2 =N_{\alpha,\beta}^2+\la\alpha,\beta\ra$, we obtain that condition \eqref{skt1} for $\alpha,\beta\in\Delta^+$, $\alpha\ne\beta$ is equivalent to 
\begin{align*}
\la\alpha,\beta\ra=& 
N_{\alpha,\beta}^2(x_{\alpha+\beta}-x_{\alpha}-x_{\beta} +x_{\alpha-\beta}+\epsilon_{\alpha-\beta}(-x_{\alpha}+x_\beta)) 
+\la\alpha,\beta\ra (x_{\alpha-\beta}+\epsilon_{\alpha-\beta}(-x_{\alpha}+x_\beta)) \\ 
=&N_{\alpha,\beta}^2(-1 +x_{\alpha-\beta} - (x_{\alpha-\beta}-1)) 
+\la\alpha,\beta\ra (x_{\alpha-\beta}- (x_{\alpha-\beta}-1))= \la\alpha,\beta\ra, 
\end{align*}
so it holds.  Concerning condition \eqref{skt2}, we have that 
\begin{align*}
x_\alpha+x_\beta+x_\gamma+x_\delta-2x_{\alpha+\beta} 
= x_{\alpha+\beta}+1 +x_{\gamma+\delta}+1 -2x_{\alpha+\beta} =2. 
\end{align*}
If $\alpha+\gamma\in\Delta^+$, then by using that $\alpha=(\alpha+\gamma)+(-\gamma)$ and $-\delta=\beta+(-(\beta+\delta))$, we obtain
\begin{align*}
\epsilon_{\alpha+\gamma}(-x_\alpha+x_\beta+x_\gamma-x_\delta+2\epsilon_{\alpha+\gamma}x_{\alpha+\gamma}) 
%=& -x_\alpha+x_{\gamma}-x_{\delta}+x_\beta+2x_{\alpha+\gamma} \\ 
= -x_{\alpha+\gamma}+1-x_{\delta+\beta}+1+2x_{\alpha+\gamma} = 2.
\end{align*}
Analogously, it is equal to $-2$ whenever $\alpha+\gamma\in\Delta^-$.  On the other hand, if $\alpha+\delta\in\Delta^+$, then 
\begin{align*}
\epsilon_{\alpha+\delta}(-x_\alpha+x_\beta-x_\gamma+x_\delta+2\epsilon_{\alpha+\delta}x_{\alpha+\delta}) 
%=& -x_\alpha+x_\beta-x_\gamma+x_\delta+2x_{\alpha+\delta} \\ 
%=& -x_\alpha+x_{\delta}+x_\beta-x_{\gamma}+2x_{\alpha+\delta} \\
= -x_{\alpha+\delta}+1-x_{\beta+\gamma}+1+2x_{\alpha+\delta} = 2,  
\end{align*}
and it is equal to $-2$ if $\alpha+\delta\in\Delta^-$.  This implies that \eqref{skt2} is equivalent to 
\begin{align*}
N_{\alpha,\beta} N_{\gamma,\delta}-N_{\alpha,\gamma} N_{\beta,\delta}+N_{\alpha,\delta}  N_{\beta,\gamma}=0,
\end{align*}
which is well known to hold (see \S\ref{roots}).  Part (i) therefore holds, concluding the proof.  
\end{proof}

%\begin{example}\label{An-SKT} 
%For $M=\SU(n+1)$, the metrics given in Theorem \ref{SKT2}, (ii) are given by $g_\tg=\gk|_\tg$ and 
%$$
%x_{\alpha_{i,i+1}}=x_i, \quad \forall i=1,\dots,n, \qquad 
%x_{\alpha_{ij}}=x_i+\dots+x_{j-1}+i-j+1, \quad\forall i<j, \quad j-i\geq 2.
%$$
%\end{example}

\begin{remark}
The above proof shows that condition \eqref{skt2} actually follows from \eqref{skt1} in the characterization of pluriclosed given in Corollary \ref{SKT}.  
\end{remark}

\section{CYT metrics}\label{CYT-sec}

On any Hermitian manifold $(M^{2n},J,g,\omega)$, there is a line $\nabla^t$, $t\in\RR$ of {\it canonical} connections which are Hermitian (i.e., $\nabla^tJ=0$ and $\nabla^tg=0$) and with torsion first studied by Gauduchon in \cite{Gdc} and defined by
$$
g(\nabla^t_XY,Z)=g(\nabla^g_XY,Z)+\tfrac{t-1}{4}d^c\omega(X,Y,Z)+\tfrac{t+1}{4}d^c\omega(X,JY,JZ),
$$
where $\nabla^g$ is the Levi-Civita connection of $(M,g)$.  Their corresponding (first) Ricci forms, given by 
$$
\rho^t(X,Y):=-\unm\sum_{i=1}^{2n} g(R^t(X,Y)e_i,Je_i), 
$$
where $\{ e_i\}_1^{2n}$ is a $g$-orthonormal frame and $R^t(X,Y):=[\nabla^t_X,\nabla^t_Y]-\nabla^t_{[X,Y]}$ is the curvature of $\nabla^t$, are all closed $2$-forms representing the first Chern class $c_1(M,J)$ up to scaling. 

We are interested in this paper in the {\it Chern} connection $\nabla^C:=\nabla^1$ and the {\it Bismut} (or {\it Strominger}) connection $\nabla^B:=\nabla^{-1}$, whose Ricci forms are related by (see \cite[(2.10)]{IvnPpd}) 
%On a given Hermitian manifold $(M,J,g,\omega)$, the {\it Chern} and {\it Bismut} (or Strominger) connections are the Hermitian connections (i.e., $\nabla J=0$ and $\nabla g=0$) respectively defined by 
%$$
%g(\nabla^C_XY,Z)=g(\nabla^g_XY,Z)+\tfrac{1}{2}d^c\omega(X,JY,JZ), \quad
%g(\nabla^B_XY,Z)=g(\nabla^g_XY,Z)-\tfrac{1}{2}d^c\omega(X,Y,Z),
%$$
%where $\nabla^g$ is the Levi-Civita connection.  The corresponding Ricci forms are both closed $2$-forms representing (up to scaling) the first Chern class $c_1(M,J)$ and 
%satisfy that (see \cite[(2.10)]{IvnPpd} or \cite[Section 4]{FinGrn1})
\begin{equation}\label{rhos}
\rho^B=\rho^C-dd_g^*\omega.  
\end{equation}
The metric $g$ is called {\it Calabi-Yau with torsion} (CYT for short) if $\rho^B=0$, i.e., the Bismut Ricci form vanishes.  
  
Back in the setting of left-invariant Hermitian structures on Lie groups, according to \cite[(4.2)]{Vzz} (see also \cite[(2.2)]{CRF} and \cite[Section 5]{GstPds}), the Chern Ricci form is given by
$$
\rho^C(X,Y)=-\unm\tr{J\ad{[X,Y]}}, \qquad\forall X,Y\in\ggo.  
$$ 
This implies that on compact semisimple Lie groups,  
\begin{align*}
\rho^C(X,Y) =& -\unm \sum_{\alpha\in\Delta^+} 2\im\alpha([X,Y]_\tg)  
=- \sum_{\alpha\in\Delta^+} \im\kil([X,Y],H_\alpha) \\
=& - \sum_{\alpha\in\Delta^+} \kil_\ggo([X,Y],\im H_\alpha)
=-\kil_\ggo([\cdot,\cdot],Z_{J_\qg}), 
\end{align*}
where $Z_{J_\qg}$ is as in \eqref{koszul-g} (see also \cite[Section 4]{FinGrn1} and \cite{AlkPrl, Ksz}).  We therefore obtain from Lemma \ref{dto} and \eqref{rhos} that 
\begin{equation}\label{rhoB}
\rho^B= \kil_\ggo\left([\cdot,\cdot],-Z_{J_\qg} + P_\tg Z_{J_\qg,g_\qg}\right), \qquad
g_\tg=\gk(P_\tg\cdot,\cdot).  
\end{equation}
In particular, the metric $g=g_\tg+g_\qg$ on the complex manifold $(G,J)$ is CYT if and only if $P_\tg Z_{J_\qg,g_\qg}=Z_{J_\qg}$.   We are now ready to give, as an application of Theorem \ref{SKT2}, an alternative proof for \cite[Corollary 5.1, iii)]{FinGrn1}.  

\begin{proposition}\label{pluri-CYT}\cite{FinGrn1} 
A left and $\Ad(T)$-invariant Hermitian structure $(J,g)$ on a compact semisimple Lie group $G$ is pluriclosed and CYT if and only if $g$ is a bi-invariant metric (i.e., Bismut flat).
\end{proposition}

\begin{proof}
Recall from Theorem \ref{SKT2} the form of any left and $\Ad(T)$-invariant pluriclosed metric $g$.  We can therefore assume that $G$ is simple and $P_\tg=I$, which implies that $g$ is CYT if and only if $Z_{J_\qg,g_\qg}=Z_{J_\qg}$ (see \eqref{rhoB}), that is, 
$$
\sum_{\alpha\in\Delta^+} \left(1-\tfrac{1}{x_\alpha}\right)\im H_\alpha = 0, \quad 
\mbox{where}\quad x_\alpha=1+\sum_{i=1}^n k^\alpha_i(x_i-1) \quad\mbox{if} \quad 
\alpha=\sum_{i=1}^n k^\alpha_i\alpha_i.  
$$
Using that $\{ \im H_{\alpha_1},\dots,\im H_{\alpha_n}\}$ is a basis of $\tg$, we obtain that this is equivalent to 
\begin{equation}\label{cyt1}
\sum_{\alpha\in\Delta^+} \left(1-\tfrac{1}{x_\alpha}\right)k^\alpha_j = 0, \qquad\forall j=1,\dots,n. 
\end{equation}
We need to show that this implies that $x_1=\dots=x_n=1$ in order to prove the proposition.  

For each fixed nonzero vector $(y_1,\dots,y_n)$, we consider the straight line 
$$
x_i=x_i(t):=1+ty_i, \qquad i=1,\dots,n, \quad \mbox{so} \quad 
x_\alpha=x_\alpha(t)=1+\sum_{i=1}^n k^\alpha_ity_i,  
$$
and the function $f(t):=\sum\limits_{\alpha\in\Delta^+} x_\alpha - \log{x_\alpha}$.  Thus $f'(0)=0$, 
$$
f'(t)=\sum\limits_{\alpha\in\Delta^+} x_\alpha'\left(1-\tfrac{1}{x_\alpha}\right) 
=\sum\limits_{i=1}^n y_i\sum\limits_{\alpha\in\Delta^+} k^\alpha_i\left(1-\tfrac{1}{x_\alpha}\right), \qquad 
f''(t)=\sum\limits_{\alpha\in\Delta^+} \tfrac{(x_\alpha')^2}{x_\alpha^2}.
$$
Since $f''(t)>0$ for all $t$ (note that $x_{\alpha_i}'(t)=y_i$ for all $i=1,\dots,n$), we have that $f'(t)=0$ if and only if $t=0$.  But if the metric $g(x_1,\dots,x_n)$ is CYT then $f'(t)=0$ by \eqref{cyt1}, concluding the proof. 
\end{proof}

\section{Pluriclosed flow}\label{PF-sec}

On a given complex manifold $(M,J)$, the {\it pluriclosed flow} for a one-parameter family of Hermitian metrics $\omega(t)=g(t)(J\cdot,\cdot)$ is defined by 
\begin{equation}\label{PF} 
\dpars\omega=-(\rho^B)^{1,1}, \qquad \omega(0)=\omega_0,
\end{equation}
where $(\rho^B)^{1,1}$ is the $(1,1)$-component of the Bismut Ricci form $\rho^B$ (see \eqref{rhos}).  

According to Theorem \ref{SKT2}, any left and $\Ad(T)$-invariant pluriclosed metric $g$ on a given compact semisimple Lie group $G$ endowed with a Samelson complex structure $J$ is of the form $g=g_\tg+g_\qg$, and up to scaling, on each simple factor of $G$ we have that $g_\tg=\gk|_{\tg}$, $g_\qg=(x_\alpha)_{\alpha\in\Delta^+}$ and if $\Pi=\{\alpha_1,\dots,\alpha_n\}$ and we set $x_i:=x_{\alpha_i}$, then 
\begin{equation}\label{xalfa}
x_\alpha=1+\sum_{i=1}^n k^\alpha_i(x_i-1), \qquad\forall 
\alpha=\sum_{i=1}^n k^\alpha_i\alpha_i.   
\end{equation}
It follows from \eqref{rhoB} that the pluriclosed flow evolution for $g(t)=g_\tg(t)+g_\qg(t)$, $g_\qg(t)=(x_\alpha(t))_{\alpha\in\Delta^+}$ is given on each simple factor of $G$ by $g_\tg(t)=\gk|_{\tg}$ for all $t$ and for each $\beta\in\Delta^+$,
\begin{align*}
-\im x_\beta' =& y_\beta'=\ddt\omega_\qg(E_\beta,E_{-\beta})=-\rho^B(E_\beta,E_{-\beta}) = \kil_\ggo(H_\beta, Z_{J_\qg}-Z_{J_\qg,g_\qg}) \\ 
&= -\im\kil_\ggo\left(\im H_\beta, \sum_{\alpha\in\Delta^+} \left(1-\tfrac{1}{x_\alpha}\right)\im H_\alpha\right) 
=\im\sum_{\alpha\in\Delta^+} \left(1-\tfrac{1}{x_\alpha}\right) \gk(\im H_\beta,\im H_\alpha),
\end{align*}
which is equivalent to 
\begin{equation}\label{PF-LG} 
x_\beta'=-\sum_{\alpha\in\Delta^+} \left(1-\tfrac{1}{x_\alpha}\right) \la\alpha,\beta\ra, \qquad\forall\beta\in\Delta^+.     
\end{equation}
By \eqref{xalfa}, the pluriclosed flow \eqref{PF} is therefore equivalent to the following ODE system: 
\begin{equation}\label{PF-LG2} 
x_j'
%=-\sum_{\alpha\in\Delta^+} \left(1-\tfrac{1}{x_\alpha}\right) \la\alpha,\alpha_j\ra 
=-\sum_{\alpha\in\Delta^+,1\leq i\leq n} \left(1-\tfrac{1}{x_\alpha}\right) k^\alpha_i \la\alpha_i,\alpha_j\ra, \qquad\forall j=1,\dots,n.     
\end{equation}
%and also to 
%\begin{equation}\label{PF-LG3} 
%x_j'=-\sum_{\substack{\alpha\in\Delta^+\\ 1\leq i\leq n}} \left(1-\tfrac{1}{\sum\limits_{l=1}^n k^\alpha_lx_l}\right) k^\alpha_i \la\alpha_i,\alpha_j\ra, \qquad\forall j=1,\dots,n.     
%\end{equation}
We consider the convex functional already used in \S\ref{CYT-sec} given by 
$$
F(X):=\sum_{\alpha\in\Delta^+} x_\alpha-\log{x_\alpha}, \qquad \forall X:=(x_1,\dots,x_n), 
$$
which has
$$ 
\tfrac{\partial}{\partial x_i}F=\sum_{\alpha\in\Delta^+} k^\alpha_i\left(1-\tfrac{1}{x_\alpha}\right), \qquad\forall i=1,\dots,n.  
$$
See Remark \ref{xpos} about the positivity region of the $x_\alpha$.  In the proof of Proposition \ref{pluri-CYT}, we showed that $F(X)>F(X_0)$ for all $X\ne X_0:=(1,\dots,1)$.  

It follows from \eqref{cyt1} and Proposition \ref{pluri-CYT} that $\grad(F)_X=0$ if and only if $X=X_0$, where $\grad(F)_X=(\tfrac{\partial}{\partial x_1}|_XF,\dots,\tfrac{\partial}{\partial x_n}|_XF)$ is the gradient of $F$ at the point $X$.  
Using \eqref{PF-LG2}, we obtain that the pluriclosed flow can be written as
\begin{equation}\label{PF-LG4} 
X'(t)=-Q\grad(F)_{X(t)}, \qquad \mbox{where}\quad Q:=[\la\alpha_i,\alpha_j\ra].    
\end{equation}
Note that $Q$ is a positive definite symmetric $n\times n$ matrix.  Thus
$$
\ddt F(X(t)) = \la X'(t),\grad(F)_{X(t)}\ra = -\la Q\grad(F)_{X(t)},\grad(F)_{X(t)}\ra \leq 0, 
$$
where equality holds if and only if $\grad(F)_{X(t)}=0$ as $Q$ is positive definite, if and only if $X(t)\equiv X_0$.  In other words, $F$ is a Liapunov function for the dynamical system and it follows from standard methods that $X_0$ is asymptotically stable, in the sense that any solution $X(t)$ starting near $X_0$ converges to $X_0$ as $t\to\infty$ (see e.g.\ \cite[Section 2.9]{Prk}).  On the other hand, if we change to the variable $Y:=Q^{-\unm}X$ and consider the new functional $H(Y):=F(Q^{\unm}Y)$, then it is easy to check that $X(t)$ is a solution to \eqref{PF-LG4} if and only if $Y(t)=Q^{-\unm}X(t)$ is the negative gradient flow of $H$:
\begin{equation*} 
Y'(t)=-\grad(H)_{Y(t)}.    
\end{equation*}
Since $Y_0:=Q^{-\unm}X_0$ is a global minimum and only critical point of $H$ and the function $H$ is real analytic, it is well known that indeed any solution $Y(t)$ converges to $Y_0$ as $t\to\infty$.   

We therefore obtain that any pluriclosed solution $g(t)$ converges to a bi-invariant metric.  More precisely, we have proved the following.  

\begin{theorem}\label{PF-thm}
The pluriclosed solution $g(t)$ on a compact semisimple Lie group $G=G_1\times\dots\times G_s$ starting at any left and $\Ad(T)$-invariant Hermitian structure $(J,g(0))$, say $g_\tg(0)(\tg_i,\tg_j)=0$ for all $i\ne j$, $g_\tg(0)|_{\tg_i}=z_i\gk|_{\tg_i}$ and $g_\qg(0)|_{\qg_i}=(z_ix_\alpha)_{\alpha\in\Delta^+_i}$ as in Theorem \ref{SKT2}, converges to the bi-invariant metric $g_b=z_1(-\kil_{\ggo_1})+\dots+z_s(-\kil_{\ggo_s})$ as $t\to\infty$.
\end{theorem}

Recall that in the case when $(G,J)$ is irreducible as a complex manifold, the limit $g_b$ is the unique bi-invariant metric (or Bismut flat) up to scaling compatible with $J$.  

\begin{example}
Using \eqref{PF-LG2} and \cite[Chapter X, \S 3]{Hlg} or \cite[Appendix C, \S 1,2]{Knp}, it is straightforward to see that the pluriclosed flow $(x(t),y(t))=(x_1(t),x_2(t))$ on $G=\SU(3)$ is given by 
$$
\left\{\begin{array}{l} 
x'= \frac{2}{x}-\frac{1}{y}+\frac{1}{x+y-1}-2, \\ y'=-\frac{1}{x}+\frac{2}{y}+\frac{1}{x+y-1}-2. 
\end{array}\right. 
$$
Analogously, for $G=\SO(5)$ we obtain that
$$
\left\{\begin{array}{l} 
x'= \frac{2}{x}-\frac{1}{y}+\frac{1}{x+y-1}-2, \\ y'=-\frac{1}{x}+\frac{1}{y}+\frac{1}{x+2y-2}-1, 
\end{array}\right. 
$$
and for $G=\Gg_2$, 
$$
\left\{\begin{array}{l} 
x'= \frac{2}{x}-\frac{3}{y}-\frac{1}{x+y-1}+\frac{1}{2x+y-2}+\frac{3}{3x+y-3}-2, \\ y'=-\frac{3}{x}+\frac{6}{y}+\frac{3}{x+y-1}-\frac{3}{3x+y-3}+\frac{3}{3x+2y-4}-6.   
\end{array}\right.
$$
Note that the lines $x=1$ and $y=1$ are both invariant in the three cases.  
\end{example}

\section{Appendix: Roots}\label{roots}

The following basic properties of root systems can be found in, among many other books, \cite[Chapter III, \S4,5, Chapter X, \S3]{Hlg} and \cite[Chapter II, \S5]{Knp}.   

Given a connected and compact semisimple Lie group $G$ and a maximal torus $T\subset G$, we have that the Lie algebra $\tg$ of $T$ is a maximal abelian subalgebra of $\ggo$, the Lie algebra of $G$, and the complexification $\tg^c$ is a Cartan subalgebra of $\ggo^c$.  The adjoint action of $\tg^c$ on the complex semisimple Lie algebra $\ggo^c$ determines the root space decomposition:
$$
\ggo^c=\tg^c\oplus\bigoplus_{\alpha\in\Delta}\ggo_\alpha, 
$$
where $\Delta\subset(\tg^c)^*$ is the {\it root system}, i.e., $[H,E]=\alpha(H)E$ for all $H\in\tg^c$ and $E\in\ggo_\alpha$.  Since the Killing form $\kil$ of $\ggo^c$ is non-degenerate on $\tg^c$, for any $\alpha\in\Delta$, there is a unique $H_\alpha\in\tg^c$ such that $\alpha=\kil(H_\alpha,\cdot)$.  There exist vectors $E_\alpha$, $\alpha\in\Delta$ such that  
$$
\ggo_\alpha=\CC E_\alpha, \qquad \kil(E_\alpha,E_{-\alpha})=1, \qquad H_\alpha=[E_\alpha,E_{-\alpha}], \qquad 
[E_\alpha,E_\beta]=N_{\alpha,\beta}E_{\alpha+\beta}. 
$$
If $\Delta^+$ is an {\it ordering} of $\Delta$ (i.e., $\Delta=\Delta^+\cup -\Delta^+$ (disjoint union) and $(\Delta^++\Delta^+)\cap\Delta\subset\Delta^+$) and $\Pi\subset\Delta^+$ is the subset of {\it simple} roots (i.e., those which are not the sum of two positive roots), then the set  
$$
\{H_\alpha:\alpha\in\Pi\} \cup \{E_\alpha:\alpha\in\Delta\}
$$ 
is a basis of $\ggo^c$ called the {\it Chevalley} basis and the set 
$$
\{\im H_\alpha:\alpha\in\Pi\} \cup \left\{e_\alpha:=\tfrac{1}{\sqrt{2}}(E_\alpha-E_{-\alpha}), \quad f_\alpha:=\tfrac{\im}{\sqrt{2}}(E_\alpha+E_{-\alpha})\right\} 
$$
is a basis of $\ggo$ called the {\it Cartan-Weyl} basis.

Every $\beta\in\Delta^+$ can be written in a unique way as $\beta=\sum\limits_{\alpha\in\Pi} n_\alpha\alpha$, where $n_\alpha\in\ZZ_{\geq 0}$ for any $\alpha\in\Pi$, and there exists a unique {\it maximal} root $\alpha_{max}$, i.e., for any $\alpha\in\Delta^+$, either $\alpha_{\max}-\alpha\in\Delta^+$ or $\alpha_{\max}-\alpha\notin\Delta$.  Since $H_\alpha\in\im\tg$ for any $\alpha\in\Delta$ and $\kil$ is positive definite on $\im\tg$, we can define 
$$
\la\alpha,\beta\ra:=\kil(H_\alpha,H_\beta) = -\kil_\ggo(\im H_\alpha,\im H_\beta), 
$$ 
where $\kil_\ggo$ is the Killing form of $\ggo$.  

The vectors $E_\alpha$ can be chosen in such a way that the numbers $N_{\alpha,\beta}$ such that $[E_\alpha,E_\beta]=N_{\alpha,\beta}E_{\alpha+\beta}$ satisfy that (see \cite[Chapter VI, Theorem 6.6]{Knp})
$$
N_{-\alpha,-\beta}=-N_{\alpha,\beta}.   
$$ 
These numbers also satisfy the following properties:
$$
N_{\beta,\alpha}=-N_{\alpha,\beta}, \qquad N_{\alpha,\beta}=N_{\beta,\gamma}=N_{\gamma,\alpha}, \quad\mbox{if}\quad \alpha+\beta+\gamma=0,
$$ 
and if $\alpha+\beta+\gamma+\delta=0$ and the sum of any pair is nonzero, then  
$$
N_{\alpha,\beta}N_{\gamma,\delta}-N_{\alpha,\gamma}N_{\beta,\delta}+ N_{\alpha,\delta}N_{\beta,\gamma} =0.
$$  
Moreover, an explicit formula for $N_{\alpha,\beta}^2$ is given as follows.  The only multiples of $\alpha\in\Delta$ which are roots are $0,\pm\alpha$, and given $\alpha,\beta\in\Delta$, we have that $\beta+n\alpha\in\Delta$ if and only if $p\leq n\leq q$, where $p\leq 0\leq q$ and 
$$
p+q=-\tfrac{2\la\alpha,\beta\ra}{\la\alpha,\alpha\ra}\in\ZZ, \qquad
N_{\alpha,\beta}^2=\unm q(1-p)\la\alpha,\alpha\ra.  
$$
In particular, $N_{\alpha,-\beta}^2=N_{\alpha,\beta}^2+\la\alpha,\beta\ra$, so $N_{\alpha,\beta}^2=-\la\alpha,\beta\ra$ if $\alpha+\beta\in\Delta$ and $\alpha-\beta\notin\Delta$ (e.g., whenever $\alpha,\beta\in\Pi$).

\end{document}